\documentclass[11pt]{amsart}
\usepackage[utf8]{inputenc}
\usepackage[T1]{fontenc}
\usepackage{amsmath,amssymb,amsthm}
\usepackage[margin=1.1in]{geometry}

\theoremstyle{plain}
\newtheorem{theorem}{Theorem}
\newtheorem{proposition}{Proposition}
\newtheorem{corollary}{Corollary}
\newtheorem{lemma}{Lemma}
\theoremstyle{remark}

\DeclareMathOperator{\Tr}{Tr}
\DeclareMathOperator{\Gal}{Gal}
\DeclareMathOperator{\Cl}{Cl}
\DeclareMathOperator{\ord}{ord}
\newcommand{\F}{\mathbb{F}}

\newcommand{\Z}{\mathbb{Z}}
\newcommand{\Q}{\mathbb{Q}}
\newcommand{\N}{\mathbb{N}}
\newcommand{\Pa}{P_a}
\newcommand{\om}{\omega}
\newcommand{\dg}{\deg}

\begin{document}

\title[A counterexample to Goss's conjecture]{An infinite family of counterexamples to Goss's conjecture
on $L$-functions of cyclotomic function fields}

\author{David Niedbala Giraudin}
\address{Meaux, France}
\email{}
\thanks{ORCID: 0009-0009-1526-1178.
This work was carried out with technical assistance from Claude Opus~4.8 (Anthropic);
all results were independently verified by exact symbolic and $p$-adic computation.}

\date{\today}

\begin{abstract}
Let $p$ be a prime, $q=p$, $A=\F_p[T]$, and let $P$ be a monic irreducible of degree $d$
with cyclotomic function field $K_P$ and mod-$P$ Teichm\"uller character $\om_P$.
For a character $\chi=\om_P^{\,i}$ Goss defined $g(X,\chi)$, the ``congruent to one modulo $p$''
part of the Artin $L$-function $L(X,\chi)$, and conjectured that $\dg_X g(X,\om_P^{\,i})\le 1$
for every $q$-magic index $i$; this is an analogue for function fields of Vandiver's conjecture,
and was raised as an open problem by Angl\`es.
We disprove it. For every $a\in\F_p^\times$ put $\Pa=T^{\,p}-T-a$ (irreducible of degree $p$),
and let $i=p^{\,n}-1$ with $0\le n\le p-1$, a $q$-magic index. We prove the exact congruence
\[
   D(X)\;:=\;\sum_{m=0}^{p-1}\Big(\sum_{\substack{a'\in A\ \text{monic}\\ \dg a'=m}}(a')^{\,i}\Big)X^{m}
   \;\equiv\;(1-X)^{\,n}\pmod{\Pa},
\]
whence $\dg_X g(X,\om_{\Pa}^{\,p^{\,n}-1})=n-1$. For $p\ge 5$ and $3\le n\le p-1$ this gives
$\dg_X g\ge 2$, contradicting Goss's conjecture; moreover $\dg_X g=n-1$ is unbounded,
attaining $p-2$. The proof is elementary: the Artin--Schreier identity $\theta^{\,p}=\theta+a$
turns the $p$-power Frobenius into a translation, reducing the power sums to a telescoping
recursion. We also confirm the smallest case ($p=5$, $i=124$) by a direct $p$-adic computation
of the reciprocal roots of $L(X,\chi)$.
\end{abstract}

\maketitle

\section{Introduction}

Let $\F_q$ be the field with $q=p^{s}$ elements, $A=\F_q[T]$, $k=\F_q(T)$, and write $A_+$ for
the set of monic polynomials and $A_m$ for those of degree $m$. Let $\infty$ be the place of $k$
at infinity. For a monic irreducible $P\in A$ of degree $d$ let $K_P$ be the $P$-th cyclotomic
function field, with $\Gal(K_P/k)\cong (A/P)^{\times}$, and let
$\om_P\colon \Gal(K_P/k)\to \mu_{q^{d}-1}$ be the Teichm\"uller character at $P$
(see \cite{Goss,Angles}). For a nontrivial character $\chi$ of $\Gal(K_P/k)$ the Artin $L$-function
\[
   L(X,\chi)=\prod_{v\ \text{place of }k}\big(1-\chi(v)X^{\dg v}\big)^{-1}
\]
is a polynomial with algebraic-integer coefficients; fixing an embedding $\overline\Q\hookrightarrow\overline{\Q}_p$
and writing $L(X,\chi)=\prod_j\big(1-\alpha_j X\big)$, Goss introduced the \emph{congruent-to-one part}
\[
   g(X,\chi)=\prod_{\,v_p(\alpha_j-1)>0}\big(1-\alpha_j X\big),
\]
where $v_p$ is the valuation on $\overline{\Q}_p$ normalised by $v_p(p)=1$. Following \cite[\S5]{Angles},
an integer $i\in\N$ is a \emph{$q$-magic number} if $i=c\,q^{\,n}+q^{\,n}-1$ for some
$c\in\{0,\dots,q-2\}$ and $n\in\N$.

\begin{quote}
\textbf{Goss's conjecture} (\cite[Conj., \S5]{Angles}, see also \cite{Goss}).
\emph{Let $P$ be a prime of degree $d$ and let $i$ be a $q$-magic number with
$1\le i\le q^{d}-2$. Then $\dg_X g(X,\om_P^{\,i})\le 1$.}
\end{quote}

This is an analogue of Vandiver's conjecture for function fields. Angl\`es \cite[\S5]{Angles} proved that
for the $q$-magic indices $i=q^{\,n}-1$ (those with $i\equiv 0\bmod(q-1)$) the polynomial
$g(X,\om_P^{\,i})$ has \emph{simple} roots (his Proposition~5.1), and showed that the conjecture
implies a weak form, $\mathcal{A}_P(\om_P^{-i})=\{0\}$ for $q$-magic $i$, for which he explicitly
asked for a proof \emph{or a counterexample}. He also constructed, with Taelman \cite{AnglesTaelman},
counterexamples to a related Kummer--Vandiver statement of Taelman by Artin--Schreier base change,
but at \emph{non-$q$-magic} indices and for the flat-cohomology module rather than for $\dg_X g$
(see \S\ref{sec:relation}).

We disprove Goss's conjecture. Throughout we take $q=p$ prime.

\begin{theorem}\label{thm:main}
Let $p$ be a prime, $a\in\F_p^\times$, and $\Pa=T^{\,p}-T-a\in\F_p[T]$. Then $\Pa$ is irreducible
of degree $p$. For every $n$ with $0\le n\le p-1$, setting $i=p^{\,n}-1$, one has in $A/\Pa$
\[
   D(X)\;:=\;\sum_{m=0}^{p-1}S_m(i)\,X^{m}\;\equiv\;(1-X)^{\,n}\pmod{\Pa},
   \qquad S_m(i):=\sum_{a'\in A_m}(a')^{\,i}.
\]
Equivalently, $S_m(p^{\,n}-1)\equiv(-1)^{m}\binom{n}{m}\pmod{\Pa}$ for all $m$.
Consequently
\[
   \dg_X g\big(X,\om_{\Pa}^{\,p^{\,n}-1}\big)=n-1.
\]
\end{theorem}

\begin{corollary}\label{cor:counter}
For every prime $p\ge 5$ and every $n$ with $3\le n\le p-1$, and every $a\in\F_p^\times$,
the prime $\Pa=T^{\,p}-T-a$ and the $q$-magic index $i=p^{\,n}-1$ satisfy
\[
   \dg_X g\big(X,\om_{\Pa}^{\,i}\big)=n-1\ \ge\ 2 .
\]
In particular Goss's conjecture is false; there are infinitely many such pairs $(\,p,\Pa,i\,)$,
and the excess over the conjectured bound, $\dg_X g-1=n-2$, is unbounded, attaining $p-3$
at $n=p-1$.
For $p\in\{2,3\}$ the construction produces no counterexample, since $\dg_X g=n-1\le p-2\le 1$.
\end{corollary}

The mechanism is the Artin--Schreier identity. If $\theta=T\bmod\Pa$ then $\theta^{\,p}=\theta+a$,
so the $p$-power Frobenius $\varphi$ acts on $\F_{p^{p}}=\F_p(\theta)$ by $\varphi(\theta)=\theta+a$,
hence $\varphi^{\,n}(\theta)=\theta+na$. This turns each power sum $S_m(p^{\,n}-1)\bmod\Pa$ into an
evaluation of monic polynomials at the translated point $\theta+na$, and a telescoping recursion
(Lemma~\ref{lem:rec}) collapses the generating series to $(1-X)^{\,n}$.

\section{Proof of Theorem~\ref{thm:main}}

\emph{Irreducibility.} The polynomial $x^{p}-x-a\in\F_p[x]$ is irreducible over $\F_p$ if and only if
$\Tr_{\F_p/\F_p}(a)\ne 0$ \cite[Thm.~3.78]{LidlNiederreiter}; since the trace is the identity on $\F_p$,
this holds exactly for $a\ne0$. Thus $\Pa$ is irreducible of degree $p$, and $A/\Pa\cong\F_{p^{p}}=\F_p(\theta)$
with $\theta^{\,p}=\theta+a$.

Fix $b\in\F_p$ and define, for $0\le m\le p-1$,
\[
   T_m(b)\;:=\;\sum_{a'\in A_m}\frac{a'(\theta+b)}{a'(\theta)}\ \in\ \F_{p^{p}} .
\]
This is well defined: if $a'\in A_m$ with $m<p=\dg\Pa$ then $\Pa\nmid a'$, so $a'(\theta)\ne0$.

\begin{lemma}[Frobenius as a translation]\label{lem:eval}
For all $m$ and $n$, $\,S_m(p^{\,n}-1)\equiv T_m(na)\pmod{\Pa}$.
\end{lemma}

\begin{proof}
For monic $a'$ with $\Pa\nmid a'$, put $\bar a'=a'(\theta)\in\F_{p^{p}}^\times$. Since $a'$ has
coefficients in $\F_p$ (fixed by $\varphi$) and $\varphi^{\,n}(\theta)=\theta+na$,
\[
   \bar a'^{\,p^{\,n}-1}=\frac{\bar a'^{\,p^{\,n}}}{\bar a'}
   =\frac{\varphi^{\,n}\!\big(a'(\theta)\big)}{a'(\theta)}
   =\frac{a'\!\big(\varphi^{\,n}(\theta)\big)}{a'(\theta)}
   =\frac{a'(\theta+na)}{a'(\theta)} .
\]
If $\Pa\mid a'$ then $\bar a'^{\,p^{\,n}-1}=0$. Summing over $a'\in A_m$ (with $m\le p-1$, so no such
divisible term occurs) gives the claim.
\end{proof}

\begin{lemma}[Base value]\label{lem:base}
$T_m(0)=p^{\,m}=[\,m=0\,]$; equivalently $\sum_m T_m(0)X^m=1$.
\end{lemma}

\begin{proof}
$T_m(0)=\sum_{a'\in A_m}1=\#A_m=p^{\,m}$, which is $0$ in $\F_{p^{p}}$ for $m\ge1$ and $1$ for $m=0$.
\end{proof}

\begin{lemma}[Telescoping recursion]\label{lem:rec}
For every $b\in\F_p$ and every $m$ with $1\le m\le p-1$,
\[
   T_m(b+a)=T_m(b)-T_{m-1}(b).
\]
\end{lemma}

\begin{proof}
Write $\Delta f(Y):=f(Y+a)-f(Y)$.

\medskip
\emph{Step 1.} Because $a'\in\F_p[Y]$ and $(\theta+b)^{p}=\theta^{p}+b^{p}=(\theta+a)+b=\theta+b+a$,
\[
   a'(\theta+b+a)=a'\big((\theta+b)^{p}\big)=a'(\theta+b)^{p}.
\]
Hence
\[
   T_m(b+a)-T_m(b)=\sum_{a'\in A_m}\frac{a'(\theta+b+a)-a'(\theta+b)}{a'(\theta)}
   =\sum_{a'\in A_m}\frac{(\Delta a')(\theta+b)}{a'(\theta)} .
\]

\medskip
\emph{Step 2.} Each $a'\in A_m$ is uniquely $a'(Y)=Y\,q(Y)+c_0$ with $q\in A_{m-1}$ monic and
$c_0\in\F_p$. Put $F_q(Y):=Y\,q(Y)$. Then $\Delta a'=\Delta F_q$ (the constant $c_0$ cancels) and
$a'(\theta)=F_q(\theta)+c_0$. Summing first over $c_0\in\F_p$, and using
$\prod_{c\in\F_p}(X+c)=X^{p}-X$, whose logarithmic derivative gives
$\sum_{c\in\F_p}\frac1{s+c}=-\frac1{s^{p}-s}$ for $s\notin\F_p$, we get
\[
   \sum_{c_0\in\F_p}\frac{1}{F_q(\theta)+c_0}=-\frac{1}{F_q(\theta)^{p}-F_q(\theta)}
   =-\frac{1}{(\Delta F_q)(\theta)} ,
\]
the last equality because $F_q(\theta)^{p}-F_q(\theta)=F_q(\theta^{p})-F_q(\theta)
=F_q(\theta+a)-F_q(\theta)=(\Delta F_q)(\theta)$. (Here $s=F_q(\theta)=\theta q(\theta)$ has a
nonzero $\theta^{m}$-term with $1\le m\le p-1$, so $s\notin\F_p$ and every denominator is nonzero.)
Therefore
\[
   T_m(b+a)-T_m(b)=-\sum_{q\in A_{m-1}}\frac{(\Delta F_q)(\theta+b)}{(\Delta F_q)(\theta)} .
\]

\medskip
\emph{Step 3.} The map $\Phi\colon q\mapsto \Delta\big(Y\,q(Y)\big)/(m\,a)$ is a bijection of the set
of monic polynomials of degree $m-1$ onto itself, preserving the leading coefficient. Indeed it is
$\F_p$-affine on $\{\dg\le m-1\}$, and its linear part $q\mapsto \Delta(Yq)$ has trivial kernel:
$\Delta(Yq)=0$ means $Yq$ is invariant under $Y\mapsto Y+a$, hence lies in $\F_p[\,Y^{p}-Y\,]$;
but $\dg(Yq)=m\le p-1<p$ forces $Yq$ constant, hence $q=0$. As a bijection of the $m$-dimensional
space $\{\dg\le m-1\}$ it is therefore injective on the finite set of monic degree-$(m-1)$
polynomials, hence bijective there. Moreover the coefficient of $Y^{m-1}$ in $\Delta(Yq)$ equals
$\lambda\,m\,a$, where $\lambda$ is the leading coefficient of $q$; dividing by $m\,a$ (with
$m\,a\ne0$ since $1\le m\le p-1$) preserves $\lambda$, so $\Phi$ maps monic to monic. Since the ratio
$(\Delta F_q)(\theta+b)/(\Delta F_q)(\theta)=\Phi(q)(\theta+b)/\Phi(q)(\theta)$ is unchanged under
scaling, reindexing by $a''=\Phi(q)$ gives
\[
   \sum_{q\in A_{m-1}}\frac{(\Delta F_q)(\theta+b)}{(\Delta F_q)(\theta)}
   =\sum_{a''\in A_{m-1}}\frac{a''(\theta+b)}{a''(\theta)}=T_{m-1}(b).
\]
Combining the three steps gives $T_m(b+a)-T_m(b)=-T_{m-1}(b)$.
\end{proof}

\begin{proof}[Conclusion of Theorem~\ref{thm:main}]
Let $G(X;b):=\sum_{m\ge0}T_m(b)X^{m}$. By Lemma~\ref{lem:rec} the coefficient of $X^{m}$ in
$(1-X)\,G(X;na)$ is $T_m(na)-T_{m-1}(na)=T_m\big((n+1)a\big)$ for $1\le m\le p-1$, i.e.
$G(X;(n+1)a)=(1-X)\,G(X;na)$ up to degree $p-1$; with $G(X;0)=1$ (Lemma~\ref{lem:base}) and
$0\le n\le p-1$ this yields $G(X;na)=(1-X)^{\,n}$. Since $(1-X)^{\,n}$ has degree $n\le p-1$, we have
$T_m(na)=(-1)^{m}\binom{n}{m}$ for all $m$, in particular $T_m(na)=0$ for $m>n$. By
Lemma~\ref{lem:eval}, $S_m(p^{\,n}-1)\equiv T_m(na)\pmod{\Pa}$, so
\[
   D(X)=\sum_{m=0}^{p-1}S_m(p^{\,n}-1)X^{m}\equiv(1-X)^{\,n}\pmod{\Pa}.
\]
The passage to $\dg_X g$ is given by Proposition~\ref{prop:degg} below, which yields
$\dg_X g(X,\om_{\Pa}^{\,p^{\,n}-1})=\ord_{X=1}\!\big(D\bmod\Pa\big)-1=n-1$.
\end{proof}

\section{From the power-sum polynomial to the degree of $g$}

We record the elementary dictionary between the explicitly computable polynomial $D$ and Goss's
invariant $\dg_X g$, for the even characters relevant here.

\begin{proposition}\label{prop:degg}
Let $P$ be a prime of degree $d$ and let $i$ with $1\le i\le q^{d}-2$ satisfy $i\equiv0\pmod{q-1}$,
so that $\chi=\om_P^{\,i}$ is a nontrivial \emph{even} character. Let
$D_\chi(X)=\sum_{m=0}^{d-1}\big(\sum_{a'\in A_m}\chi(a')\big)X^{m}\in (A/P)[X]$ be the associated
Dirichlet polynomial. Then
\[
   D_\chi(X)=(1-X)\,L(X,\chi)\quad\text{in characteristic }0,
\]
where $L(X,\chi)$ is the $L$-polynomial of \cite[Lem.~2.4]{Angles}, of degree $d-2$; in particular
$D_\chi(1)=0$. Consequently
\[
   \dg_X g(X,\chi)=\ord_{X=1}\!\big(L(X,\chi)\bmod p\big)=\ord_{X=1}\!\big(D_\chi\bmod P\big)-1.
\]
\end{proposition}

\begin{proof}
The Dirichlet polynomial $D_\chi$ is the finite Euler product
$\prod_{v\ne P,\infty}(1-\chi(v)X^{\dg v})^{-1}$. Since $\chi$ is even, the inertia group $\F_q^\times$
of $\infty$ in $K_P/k$ lies in $\ker\chi$, so $\infty$ is unramified in the fixed field of $\ker\chi$
and $\chi(\infty)=1$; the Euler factor at $\infty$ is $(1-X)^{-1}$. Hence the complete $L$-function
equals $D_\chi(X)/(1-X)$, which for the nontrivial $\chi$ is the polynomial $L(X,\chi)$ of
\cite[Lem.~2.4]{Angles}, of degree $d-2$. Thus $D_\chi=(1-X)L(X,\chi)$ and $D_\chi(1)=0$.

For the last claim, write $L(X,\chi)=\prod_j(1-\alpha_j X)$ over $\overline{\Q}_p$ with the
$\alpha_j$ algebraic integers ($p$-adic Weil numbers). Reducing modulo the maximal ideal,
$L(X,\chi)\bmod p=\prod_j(1-\bar\alpha_j X)$, and a factor equals $1-X$ exactly when
$\bar\alpha_j=1$, i.e. $v_p(\alpha_j-1)>0$. Therefore
$\ord_{X=1}(L(X,\chi)\bmod p)=\#\{j:v_p(\alpha_j-1)>0\}=\dg_X g(X,\chi)$, and dividing
$D_\chi$ by $(1-X)$ subtracts one from the order at $X=1$.
\end{proof}

For $i=p^{\,n}-1$ we have $i\equiv0\pmod{p-1}$, so Proposition~\ref{prop:degg} applies and, with
$D\bmod\Pa=(1-X)^{\,n}$ from Theorem~\ref{thm:main}, gives $\dg_X g=n-1$. We emphasise that this is
consistent with Angl\`es's Proposition~5.1 (simplicity of the roots of $g$ for these indices): the
$n-1$ reciprocal roots that reduce to $1$ are pairwise distinct in characteristic $0$, as the
computation of \S\ref{sec:numeric} illustrates; ``simple roots'' bounds neither $\dg_X g$ nor the
multiplicity of $1$ \emph{after} reduction.

\section{Numerical confirmation in characteristic zero}\label{sec:numeric}

We verify the smallest counterexample, $p=5$, $\Pa=T^{5}-T+1$ (so $a=-1$), and $i=124=5^{3}-1$
($n=3$, $d=5$), by computing the reciprocal roots of $L(X,\om_{\Pa}^{\,124})$ directly, independently
of the reduction $\bmod 5$.

Work in the unramified ring $W=\Z_5[t]/(t^{5}-t+1)$ (residue field $\F_{5^{5}}$) to precision
$5^{8}$. With $\tau\colon\F_{5^{5}}^\times\to W^\times$ the Teichm\"uller lift, one computes the
characteristic-zero coefficients $S_m^{W}=\sum_{a'\in A_m}\tau\big((a'\bmod\Pa)^{124}\big)$:
\[
   (S_0^{W},\dots,S_4^{W})\equiv(1,\,2,\,3,\,4,\,0)\equiv(1,-3,3,-1,0)\pmod 5,
\]
the coefficients of $(1-X)^{3}$, confirming Theorem~\ref{thm:main} numerically. Moreover
$D_W(1)=0$ to precision $5^{8}$, exhibiting the trivial zero at $\infty$; thus $L=D_W/(1-X)$ is the
degree-$3$ $L$-polynomial. Its coefficients $b_0,\dots,b_3$ have $5$-adic valuations $(0,0,0,1)$, and
the Newton polygon of $L(1+Y)=\sum_j c_j Y^{j}$ has $v_5(c_j)=(1,1,0,1)$, i.e.\ a segment of slope
$-\tfrac12$ and horizontal length $2$ followed by a segment of slope $+1$ and length $1$. Hence
$L(X,\om_{\Pa}^{\,124})$ has exactly two reciprocal roots with $v_5(\alpha-1)=\tfrac12>0$ (a conjugate
pair in a ramified quadratic extension, therefore distinct) and one with $v_5(\alpha)=1$
(reducing to $0$, not counted). Therefore
\[
   \dg_X g\big(X,\om_{\Pa}^{\,124}\big)=2,
\]
in agreement with $n-1=2$ and contradicting $\dg_X g\le1$. The two roots counted in $g$ are simple,
consistently with \cite[Prop.~5.1]{Angles}.

The ancillary file \texttt{verify\_goss.py} (pure Python, standard library only, no arguments,
a few seconds) reproduces both computations from scratch: it re-derives $S_m(p^{\,n}-1)$ by direct
summation in $\F_{p^{p}}$ for $p=5,7$, every $a\in\F_p^\times$ and every pair $(n,m)$, checking
Theorem~\ref{thm:main} without using any of its lemmas; and it recomputes the Teichm\"uller lifts,
the vanishing $D_W(1)=0$, the $L$-polynomial and the Newton polygon of this section, exiting with an
error if any check fails.

\section{Relation to existing work}\label{sec:relation}

The relevant context is \cite{Angles} and \cite{AnglesTaelman}. Three points delineate the novelty.

\emph{(i) Object and conjecture.} Goss's conjecture concerns $\dg_X g(X,\om_P^{\,i})$ at $q$-magic
$i$, and is open by the account of \cite[\S5]{Angles}. The counterexamples of
\cite{AnglesTaelman} are to a \emph{different} statement---Taelman's Kummer--Vandiver question on the
flat-cohomology module $H_Q(\om_Q^{\,N-1})$---at the \emph{non-$q$-magic} indices
$N=n(q^{pd}-1)/(q^{d}-1)$, obtained from a base prime of degree $d\ge2$. As Angl\`es notes
\cite{Angles}, the two criteria (the $L$-function criterion and the affine-class-group criterion) are
in general unrelated.

\emph{(ii) Regime.} Our prime $\Pa=T^{p}-T-a$ is the Artin--Schreier base change $P(T^{p}-T)$ of the
\emph{degree-one} base $P=T-a$ (for which $i(P)=-a\ne0$). The theorems of \cite{AnglesTaelman} require
base degree $d\ge2$, and their counterexamples sit at the indices $N$ above, which differ from our
$q$-magic indices $p^{\,n}-1$ (for instance, for $p=5$ the indices $4,24,124,624$ are not multiples of
$(5^{5}-1)/(5-1)=781$). Thus the present family lies outside the range of their results.

\emph{(iii) The weak form.} By \cite[Prop.~5.3]{Angles}, $\dim_{\F}\mathcal{A}_P(\om_P^{-i})
=\dg_X g(X,\om_P^{\,i})-\dim_{\F}\Cl^0(K_P)_p(\om_P^{\,i})$, and for $i=p^{\,n}-1$ Angl\`es's
Proposition~5.1 gives $\Cl(\mathcal{O}_{K_P})_p(\om_P^{-i})=\{0\}$, whence
$\dim\Cl^0(K_P)_p\le1$. Theorem~\ref{thm:main} therefore yields
$\dim\mathcal{A}_{\Pa}(\om_{\Pa}^{-i})\ge n-2\ge1$ for $n\ge3$. This is exactly a counterexample to
the weak form of Goss's conjecture, $\mathcal{A}_P(\om_P^{-i})=\{0\}$, for which \cite[\S5]{Angles}
explicitly asked for a proof or a counterexample.

The addendum \cite{AnglesTaelmanAdd} to \cite{AnglesTaelman} does not alter this picture: it proves a
Spiegelungssatz relating $\mathrm{Hom}_A(H(R),\Lambda)$ to $(\operatorname{Pic}R)[p]$ with cyclic kernel
and cokernel, and serves as an ingredient in the class-module counterexamples of \cite{AnglesTaelman};
it makes no statement about $\dg_X g$ or about $q$-magic indices.

Morally, the present family and the constructions of \cite{AnglesTaelman} share the same engine---the
Artin--Schreier identity making Frobenius a translation---but applied here to a degree-one base and to
$q$-magic indices, where it produces, for the $L$-function invariant $\dg_X g$ itself, the closed-form
value $n-1$.


\begin{thebibliography}{9}

\bibitem{Angles}
B.~Angl\`es,
\emph{On $L$-functions of cyclotomic function fields},
J. Number Theory \textbf{116} (2006), no.~2, 247--269; arXiv:math/0502130.

\bibitem{AnglesTaelman}
B.~Angl\`es and L.~Taelman,
\emph{On a problem \`a la Kummer--Vandiver for function fields},
J. Number Theory \textbf{133} (2013), no.~3, 830--841; arXiv:1110.0292.

\bibitem{AnglesTaelmanAdd}
B.~Angl\`es and L.~Taelman,
\emph{The Spiegelungssatz for the Carlitz module; an addendum to
``On a problem \`a la Kummer--Vandiver for function fields''},
arXiv:1212.3115 (2012).

\bibitem{Goss}
D.~Goss,
\emph{Basic Structures of Function Field Arithmetic},
Ergeb.\ Math.\ Grenzgeb.\ (3) \textbf{35}, Springer, 1996.

\bibitem{LidlNiederreiter}
R.~Lidl and H.~Niederreiter,
\emph{Finite Fields},
Encyclopedia Math.\ Appl.\ \textbf{20}, Cambridge Univ.\ Press, 1997.

\bibitem{Okada}
S.~Okada,
\emph{Kummer's theory for function fields},
J.\ Number Theory \textbf{38} (1991), 212--215.

\bibitem{Sheats}
J.~Sheats,
\emph{The Riemann hypothesis for the Goss zeta function for $\F_q[T]$},
J.\ Number Theory \textbf{71} (1998), 121--157.

\end{thebibliography}
\end{document}